\documentclass[12pt]{amsart}

\usepackage{amssymb}

\usepackage{enumitem}

\usepackage{graphicx}

\makeatletter
\@namedef{subjclassname@2010}{%
  \textup{2010} Mathematics Subject Classification}
\makeatother

\usepackage[T1]{fontenc}

\theoremstyle{definition}

\numberwithin{equation}{section}

\newcommand{\dom}{\operatorname{dom}}

\newcommand{\comment}[1]{}

\newcommand{\PFA}{\mathrm{PFA}}

\newcommand{\ZFC}{\mathrm{ZFC}}

\newcommand{\range}{\mathrm{range}}

\theoremstyle{plain}
\newtheorem{thm}{Theorem}[section]
\newtheorem{lem}[thm]{Lemma}
\newtheorem{prop}[thm]{Proposition}
\newtheorem{cor}[thm]{Corollary}
\newtheorem{fact}[thm]{Fact}

\theoremstyle{definition}
\newtheorem{defn}[thm]{Definition}
\newtheorem{rmk}[thm]{Remark}

\begin{document}


\baselineskip=17pt


\title[Club Minimal Kurepa Trees]{A Minimal Kurepa Tree With respect to Club Embeddings}

\author[H. Lamei Ramandi]{Hossein Lamei Ramandi}
\address{Department of Mathematics\\ University of Toronto\\
Toronto, ON, Canada, M5S 3G3}
\email{hlramandi@gmail.com}

\date{}

\begin{abstract}
We will show it is consistent with $\mathrm{GCH}$ that there is a  minimal Kurepa tree with respect to 
club embeddings. That is, there is a Kurepa tree $T$ which club embeds in all of its Kurepa 
subtrees in the sense of \cite{club_isomorphic}. Moreover, the Kurepa tree we introduce has no Aronszajn 
subtree.
\end{abstract}



\maketitle

\section{Introduction}
Similarity of $\omega_1$-trees with respect to closed unbounded subsets of $\omega_1$
was first considered by Abraham and Shelah.
\begin{thm} \cite{club_isomorphic}
$\PFA$ implies that every two normal Aronszajn trees are club isomorphic.
\end{thm}
Here two $\omega_1$-trees $S,T$ are \emph{club isomorphic}
if there is a club $C\subset \omega_1$ such that
$T\upharpoonright C$ is isomorphic to $S\upharpoonright C$.
This theorem may be regarded as an evidence that under some reasonable forcing axioms,
 Aronszajn trees behave like  non-atomic countable trees. 
Although there is such an insightful theorem regarding the 
club isomorphisms of Aronszajn trees, similar questions regarding
Kurepa trees do not seem to be addressed in the literature.
For instance, considering the 
fact that $2^{<\omega}$ is a minimal 
countable non-atomic tree, one might ask whether or not there can be Kurepa trees which are 
minimal with respect to club emebeddings. In this paper we will prove:
\begin{thm} \label{main}
It is consistent with $\mathrm{GCH}$ that there is a Kurepa tree $T$ which is club isomorphic to 
all of its downward closed everywhere Kurepa subtrees. Moreover, $T$ has no Aronszajn 
subtrees.
\end{thm}
An $\omega_1$-tree $T$ is said to be \emph{everywhere Kurepa} if for all
$x \in T$, the tree of all $y \in  T$ that are comparable with $x$ is Kurepa. Since every Kurepa tree 
contains an everywhere
Kurepa subtree, this theorem implies that the tree in the theorem is actually club minimal 
with respect to being Kurepa.  In other words, for every downward closed  Kurepa subtree $U\subset T$
there is a club $C\subset \omega_1$ and a one to one  function
$f:T\upharpoonright C \longrightarrow U\upharpoonright C$ which is level and order preserving.

The forcings we use to add embeddings are not proved to be proper, but their behavior towards
suitable models $M$ are similar to proper posets often enough.
This property of posets is called $\mathcal{E}$-completeness, and was shown to be sufficient criterion
for preserving $\omega_1$ in \cite{proper_forcing}. The notion $S$-completeness here seems to coincide with 
$\mathcal{E}$-completeness.

In section
\ref{Iteration section}, based on the work in \cite{proper_forcing}, and the notion of the proper isomorphism 
condition for proper posets we will prove the lemmas needed for certain chain conditions which are not 
included in \cite{proper_forcing}. We have also included the proof of the fact that $S$-completeness 
is preserved by countable support iterations although it is proved in \cite{proper_forcing}. 
This makes the proof of the lemmas needed for chain condition properties more clear. 
Section \ref{embedding forcing} is devoted to the proof of Theorem \ref{main}.

In this paper all trees are considered to have the property that if $s,t$ are two distinct elements of the same limit height, 
they have different sets of predecessors.
An \emph{$\omega_1$-tree} is a tree which has height $\omega_1$ and countable levels.
A chain $b \subset T$
is called a \emph{cofinal branch} of $T$ if it intersects all levels of $T$.
An $\omega_1$-tree $T$ is called \emph{Aronszajn} if it has no cofinal branches.
It is called \emph{Kurepa} if it has at least $\omega_2$ many cofinal branches.
For $C\subset \omega_1$, $T\upharpoonright C=\{t\in T:$ height of $t$ is in $C\}$.
If $S,T$ are trees, $f:T \longrightarrow S$ is called a tree embedding if for all $t,s\in T, $
$t<_T s$ iff $f(t)<_S f(s)$.
Here $\mathcal{B}(T)$ is the collection 
of all cofinal branches in $T$. 
For $b \in \mathcal{B}(T)$, $b(\alpha)$
is the element in $b$ which has
height $\alpha$, and  for $b,b'$ in $\mathcal{B}(T)$, 
$b\Delta b'$ is the minimum $\alpha \in \omega_1$ such that 
$b(\alpha)\neq b'(\alpha).$
For $Z \subset \mathcal{B}(T)$, $\alpha_Z = \sup\{b\Delta b':b,b'\in Z\}$.
By $\Omega(T)$ we mean the set of all countable $Z\subset \mathcal{B}(T)$ with the property that 
for all $t \in T_{\alpha_Z}$ there is a $b \in Z$ with $t \in b$.

\section{$S$-Completeness, Iteration and Chain Condition} \label{Iteration section}

We will work with forcings which may not be proper but up to a fixed stationary set 
they behave very much like $\sigma$-complete forcings.
In this section we provide the machinery to iterate these posets without collapsing cardinals. 
Everything in this section is built on the material in \cite{proper_forcing}. 

For a regular cardinal $\theta$, $H_\theta$ is the collection of all sets of hereditary cardinality less than $\theta$.
We assume $H_\theta$ is equipped with a fixed well ordering without mentioning it.
Assume $\mathcal{P}$ is a forcing and $\theta$ is a regular cardinal such that $\mathcal{P}$ and the powerset
of $\mathcal{P}$ are in $H_\theta$. A countable elementary submodel $N$ of $H_\theta$ is said to be \emph{suitable} for $\mathcal{P}$
if $\mathcal{P} \in N$. A decreasing sequence $\langle p_n : n \in \omega \rangle$ of elements of $\mathcal{P}\cap N$ 
is said to be \emph{$(N,\mathcal{P})$-generic} if 
for all dense subsets $D$ of $\mathcal{P}$ that are in $N$ there is an $n \in \omega$ such that $p_n \in D$.
\begin{defn}
Assume $X$ is uncountable and $S \subset [X]^\omega$ is stationary. A poset $\mathcal{P}$ is 
said to be \emph{$S$-complete}
if every descending $(M, \mathcal{P})$-generic
sequence $\langle p_n: n\in \omega \rangle$ has a lower bound, for all $M$ with $M \cap X \in S$ and
$M$ suitable for $X,\mathcal{P}$.
\end{defn} 
First note that $S$-complete forcings preserve the stationarity of all stationary
subsets of $S$. Although it is clear from the definition, we emphasize that $S$-completeness is not
stronger than properness unless $S$ contains a club. If $S$ contains a club, $S$-completeness is 
very close to being $\sigma$-closed. For instance, the forcing axiom for $S$-complete forcings is a theorem of $\ZFC$,
when $S$ contains a club in $[X]^\omega$ for some uncountable  set $X.$
The following fact vacuously follows from the definition.
\begin{fact}
 Assume $X$ is uncountable and $S\subset [X]^\omega$ is stationary. If $\mathcal{P}$ is an $S$-complete
  forcing then it preserves $\omega_1$ and adds no new countable sequences of ordinals.
\end{fact}

Now we prove that for a given stationary $S\subset [X]^\omega$ where $X$ is uncountable,
the property of being  $S$-complete is preserved by countable support iterations. 
We follow the same strategy 
as in the proof of the similar lemma for proper posets in \cite{proper_forcing}.

\begin{fact} \label{comp-suc}
Assume $S, X$ are as above, $\mathcal{P}$ is $S$-complete, and 
$\Vdash_{\mathcal{P}}``\dot{\mathcal{Q}}$ is $\check{S}$-complete''.
Then $\mathcal{P}*\dot{\mathcal{Q}}$ is $S$-complete.
\end{fact}
\begin{proof}
Assume $M$ is suitable for $\mathcal{P}*\dot{\mathcal{Q}}$ and $M\cap X \in S$. Let 
$\langle p_n*\dot{q}_n : n \in \omega \rangle$ be a descending $(M, \mathcal{P}*\dot{\mathcal{Q}})$-generic 
sequence.
Since $\langle p_n : n \in \omega \rangle$ is an $(M, \mathcal{P})$-generic sequence, it has a lower bound 
$p \in \mathcal{P}$. Moreover 
\begin{center}
$p\Vdash_{\mathcal{P}}$ 
``$\langle \dot{q}_n : n \in \omega \rangle$ is an $(M[\dot{G_{\mathcal{P}}}], \dot{\mathcal{Q}})$-generic sequence.''
\end{center} 
On the other hand, the $(M,\mathcal{P})$-generic condition $p$ forces that $M[\dot{G_{\mathcal{P}}}]\cap 
\mathbf{V}=M$
and consequently $M[\dot{G_{\mathcal{P}}}]\cap \check{X} \in \check{S}$.
So it forces that the sequence $\langle \dot{q}_n : n \in \omega \rangle$ has a lower bound as well.
 Let $\dot{q}$ be 
a $\mathcal{P}$-name for such a condition, then $p*\dot{q}$ is a lower bound for
 $\langle p_n*\dot{q}_n : n \in \omega \rangle$.
\end{proof}

\begin{lem} 
Assume $X$ is uncountable, $S\subset [X]^\omega$ is stationary, 
$\langle \mathcal{P}_i, \dot{\mathcal{Q}}_j: i\leq \delta, j < \delta \rangle$ is a countable support 
iteration of 
$S$-complete forcings, $N$ is suitable for $\mathcal{P}_\delta$, $N\cap X \in S$,
$\langle p_n: n\in \omega \rangle$ is an $(N,\mathcal{P}_\delta)$-generic descending sequence of conditions,
$\alpha < \delta$ is in $N$ and
$q\in \mathcal{P}_\alpha$ is a lower bound for $\langle p_n\upharpoonright \alpha: n \in \omega \rangle.$
Then there is a  lower bound $q'\in \mathcal{P}_\delta$  for  $\langle p_n: n\in \omega \rangle$, such that
$q'\upharpoonright \alpha = q.$
\end{lem}
\begin{proof}
We use induction on $\delta$. If $\delta$ is a successor ordinal the lemma 
follows from the induction hypothesis and the argument in the proof of the 
previous fact.  If $\delta$ is limit, let $\langle \alpha_n : n \in \omega \rangle$ be a cofinal 
sequence in $N \cap \delta$ such that $\alpha_0=\alpha$, and for all $i$, $\alpha_i \in N$.
Note that for all $i$, $\langle p_n\upharpoonright \alpha_i: n\in \omega \rangle$
is a descending $(N,\mathcal{P}_{\alpha_i})$-generic sequence.
So by the induction hypothesis there is a sequence $q_i$, $i \in \omega$, such that
\begin{itemize}
\item
$q_0=q$,
\item
$q_i \in \mathcal{P}_{\alpha_i}$ is a lower bound for $\langle p_n\upharpoonright \alpha_i: n\in \omega \rangle$, and 
\item
if $i<j $ then $ q_j\upharpoonright \alpha_i = q_i$.
\end{itemize}
Now $q'=\bigcup_{i\in \omega}q_i$ works.
\end{proof}
\begin{cor} \label{iteration}
Assume $X$ is uncountable and $S\subset[X]^\omega$ is stationary. Then $S$-completeness is 
preserved under countable support iterations.
\end{cor}

Assume $T$ is an $\omega_1$-tree and $\mathcal{P}$ is a poset. 
We are interested in circumstances in which  $\mathcal{P}$
adds no new cofinal branch or no new Aronszajn subtree to $T$.
It is a well known fact that $\sigma$-closed forcings do not add new cofinal branches to $\omega_1$-trees. A similar argument shows that 
if $X$ is uncountable, $S \subset [X]^\omega$ is stationary and $\mathcal{P}$ is an $S$-complete forcing then it does not add 
cofinal branches to $\omega_1$-trees of the ground model. We include the proof of this fact for more clarity.
Note that  no forcing can add a new cofinal branch or Aronszajn subtree to $T$
when $T$ has no Aronszajn subtree and has only countably many cofinal branches. 

\begin{lem} \label{No A subtree} 
Assume $T$ is an $\omega_1$-tree 
which has uncountably many cofinal branches and which has no Aronszajn subtree 
in the ground model $\mathbf{V}$. Also assume
 $\Omega(T)\subset [\mathcal{B}(T)]^\omega$ is stationary  and
$\mathcal{P}$ is an $\Omega(T)$-complete forcing.
Then  $T$ has no Aronszajn subtree in $\mathbf{V}^\mathcal{P}$.
\end{lem}
\begin{proof}
Assume $\dot{U}$ is a $\mathcal{P}$-name for a downward closed Aronszajn subtree of $T$.
Let $p \in \mathcal{P}$, $M$ be suitable with $M\cap \mathcal{B}(T)\in \Omega(T)$, $\delta = M\cap \omega_1$, and 
$p, \dot{U}$ be in $ M.$ 
Note that if $b,b'$ are in $\mathcal{B}(T) \cap M$ then by elementarity $b \Delta b' \in \delta$.
Also by elementarity and the fact that $T$ has uncountably many cofinal branches and no Aronszajn subtree, 
for every $\xi \in \delta$ there are $b,b'$ in $\mathcal{B}(T) \cap M$
such that $b \Delta b' > \xi$. Therefore $\alpha_{\mathcal{B}(T) \cap M} = \delta$ and  for all $t \in T_\delta$ there exists 
$b \in M \cap \mathcal{B}(T) $ such that $b(\delta)=t$. Hence 
 $\{b(\delta): b \in M\cap \mathcal{B}(T)\}=T_\delta.$
 
For all $b \in M \cap \mathcal{B}(T)$ the set $D_b$ consisting of all 
conditions $q\in \mathcal{P}$ which forces that 
$b(\check{\alpha}) \notin \dot{U}$ for some $\alpha \in \omega_1$ 
is dense and in $M$. Note that if $q \in D_b \cap M$, 
then by elementarity there exists $\alpha \in \delta$ such that 
$q \Vdash b (\check{\alpha}) \notin \dot{U}$.
 Now let $\langle p_n: n\in \omega \rangle $ be a decreasing $(M,\mathcal{P})$-generic  sequence, with 
$p_0=p,$ and $\bar{p}$ be a lower bound for this sequence.
Then $\bar{p}$ forces that $\dot{U}$ has no element in $T_\delta.$
Since $U$ is downward closed, this implies that $U$ is a countable set which is a contradiction.
\end{proof}
 \begin{lem} \label{No New branch}
 Assume $T$ is an $\omega_1$-tree, $X$ is an uncountable set, 
$S \subset [X]^\omega$ is stationary, and 
$\mathcal{P}$ is an $S$-complete forcing. Then $\mathcal{P}$ does not add new cofinal branches to $T$.
 \end{lem}
\begin{proof}
Assume $\dot{b}$ is a $\mathcal{P}$-name and $p \in \mathcal{P}$ is a condition which 
forces that  $\dot{b}$ is a new cofinal branch of $T$.
Let $M$ be suitable for 
$\mathcal{P}$ with $M\cap X \in S$ and 
$\langle D_n : n \in \omega \rangle$ be an enumeration of the dense subsets of 
$\mathcal{P}$ that are in $M.$
Also let $\delta = M\cap \omega_1$.
Inductively choose $p_\sigma$ for each $\sigma \in 2^{< \omega}$ such that:
\begin{enumerate}
\item $p_0 \leq p$,
\item if $\sigma$ is an initial segment of $\pi$ then $p_\pi  \leq p_\sigma$,
\item \label{incomparable}
 if $\sigma$ and $\pi$ are incomparable then there are incomparable $s,t$ in $T \cap M$ such that 
$p_\sigma \Vdash \check{t} \in \dot{b}$ and $p_\pi \Vdash \check{s} \in \dot{b}$, and
\item $p_\sigma$ is in $M\cap D_{|\sigma|}$.
\end{enumerate}
In order to reach a contradiction, we show that $T_\delta$ has to be uncountable. Fix $f \in 2^\omega$.
Then $\langle p_{f \upharpoonright n} : n \in \omega \rangle$ is an $(M, \mathcal{P})$-generic sequence.
Note that  for each $\xi \in \delta$ there exists $n \in \omega$ and $t \in T_\xi$ 
such that $p_{f \upharpoonright n}   \Vdash \check{t} \in T_\xi$.
In other words, $f$ determines a cofinal branch through $T_{<\delta}$ uniquely.
Since $p$ forces that $\dot{b}$ is a cofinal branch in $T$ there is a unique $t_f \in T_\delta$ which is forced to be in $\dot{b}$ by any lower bound of 
$\langle p_{f \upharpoonright n} : n \in \omega \rangle$.
On the other hand by \ref{incomparable}, if $f,g$ in $2^{< \omega}$ are distinct then $t_f \neq t_g$. 
Therefore, $T_\delta$ is uncountable.
\end{proof}
Now we deal with the chain condition issue for $S$-complete forcings. The following definition 
is a modification of the $\kappa$-properness isomorphism condition.

\begin{defn} \label{S-cic}
Assume $S,X$ are as above and $\kappa$ is a regular cardinal. 
We say that $\mathcal{P}$ satisfies the \emph{$S$-closedness isomorphism condition for 
$\kappa$},
or $\mathcal{P}$ has the \emph{$S$-cic for $\kappa$},  if whenever
\begin{itemize}
\item
$M,N$ are suitable models for $\mathcal{P}$,
\item
both $M \cap X,$ $ N\cap X$ are in $S$,
\item
$h:M\rightarrow N$ is an isomorphism such that $h\upharpoonright (M\cap N) =\mathrm{ id}_{(M \cap N)}$,
\item
$\min((N \setminus M)\cap \kappa)> \sup(M\cap \kappa)$, and 
\item
$\langle p_n: n\in \omega \rangle $ is an $(M,\mathcal{P})$-generic sequence,
\end{itemize} 
then there is a common lower bound $q \in \mathcal{P}$ for
$\langle p_n: n\in \omega \rangle $ and $\langle h(p_n): n\in \omega \rangle $.
\end{defn}
\begin{lem} \label{chain}
Assume $2^{\aleph_0} < \kappa$, $\kappa$ is a regular cardinal and $S,X$ are as above. 
If $\mathcal{P}$ satisfies the $S$-cic for $\kappa$ then it has the
 $\kappa$-c.c.
\end{lem}
\begin{proof}
Let $\langle p_\xi: \xi \in \kappa \rangle $ be a collection of conditions in $\mathcal{P}$,
and for each $\xi \in \kappa$, $M_\xi$ be a suitable model for $\mathcal{P}$ such that  $M\cap X \in S$,
$\kappa ,\xi $, and $\langle p_\xi: \xi \in \kappa \rangle $ are in $ M$. 
Consider the function $f: \kappa \longrightarrow \kappa$ defined by $\xi \mapsto \sup(M_\xi \cap \xi)$. 
Obviously for all $\xi$ with $\mathrm{cf}(\xi)>\omega$, $f(\xi)<\xi.$ So there is a stationary $W\subset \kappa$ such 
that the function $f\upharpoonright W$ is a constant. Fix $U \subset W$ of size $\kappa$ such that 
for all $\xi < \eta$ in $U$, $\sup(M_\xi \cap \kappa)< \eta$ and $M_\xi \cap \xi = M_\eta \cap \eta$.

For each $\xi \in U$, let
$\langle p_\xi^n: n \in \omega \rangle $ be descending and $(M_\xi,\mathcal{P})$-generic
 with $p_\xi^0=p_\xi$.
Since $2^{\aleph_0} < \kappa$ we can thin down $U$ if necessary so that
for all $\xi, \eta$ in $U$, $M_\xi$ is isomorphic to $M_\eta$ via the map
$h_{\xi \eta}:M_\xi \rightarrow M_\eta$ which is
induced by the transitive collapse maps. 

Now consider models $M_\xi$ together with $\langle p_\xi^n: n \in \omega \rangle $ as constants.
There are at most continuum many of the isomorphism types of these models and by extensionality the 
isomorphism between $M_\xi$ and $M_\eta$ is unique if it exists.
So we can thin down the collection  $\langle p_\xi: \xi \in U \rangle $ again,
so that
for all $\xi, \eta$ and $n \in \omega$, $h_{\xi \eta}(p_\xi^n)=p_\eta^n $.

Since $\mathcal{P}$ satisfies $S$-cic, for every pair of distinct $\xi, \eta$ in $U$, 
there is a condition $q \in \mathcal{P}$ which is a common lower bound for sequences 
$\langle p_\xi^n: n \in \omega \rangle $ and $\langle p_\eta^n: n \in \omega \rangle $.
Hence  $p_\xi$ and $p_\eta$ are compatible.
\end{proof}

We are now ready to state and prove the lemma we need for the chain condition issues. 

\begin{lem} \label{chain CS}
Suppose 
$\langle \mathcal{P}_i, \dot{\mathcal{Q}}_j: i\leq \delta, j < \delta \rangle$
is a countable support iteration of $S$-complete forcings, where $S\subset [X]^{\omega}$
is stationary and $X$ is uncountable. Assume in addition that $\kappa$ is a regular cardinal and for all $i \in \delta$,
$\Vdash_{\mathcal{P}_i} ``\dot{\mathcal{Q}}$ has the $\check{S}$-cic for $\kappa$''.
Then $\mathcal{P}_\delta$ has the $S$-cic for $\kappa$.
\end{lem}
\begin{proof}
First note that if $\mathcal{P}$ is any forcing, $M,N$ are suitable for $\mathcal{P}$, 
$h:M\rightarrow N $ is an isomorphism,
$p$ is both $(M, \mathcal{P})$-generic and $(N, \mathcal{P})$-generic,
 and $G\subset \mathcal{P}$ is $\mathbf{V}$-generic
with $p\in G$, then $h[G]:M[G] \rightarrow N[G]$ defined by $\tau_G \mapsto (h(\tau))_G$
is an isomorphism as well. 

Before we deal with the general case,
we prove the lemma for $\mathcal{P}*\dot{\mathcal{Q}}$. Let $M,N,h$ be as in Definition \ref{S-cic}
for $\mathcal{P}*\dot{\mathcal{Q}}$, and let
$\langle p_n*\dot{q}_n : n \in \omega \rangle$ be a descending $(M,\mathcal{P}*\dot{\mathcal{Q}} )$-generic
such that the sequences $\langle p_n: n\in \omega \rangle $ and $\langle h(p_n): n\in \omega \rangle $
have a common lower bound $p\in \mathcal{P}$. 
Since $p$ is both $(M, \mathcal{P})$-generic and $(N, \mathcal{P})$-generic,
it forces the hypotheses of the Definition \ref{S-cic} for 
$M[\dot{G_\mathcal{P}}],N[\dot{G_\mathcal{P}}], \dot{\mathcal{Q}}, h[\dot{G_\mathcal{P}}]$, and
$\langle \dot{q}_n : n \in \omega \rangle$. By the assumption on $\dot{\mathcal{Q}}$,
there is a $\mathcal{P}$-name $\dot{q}$ which is forced by $p$ to be a common lower bound for 
$\langle \dot{q}_n : n \in \omega \rangle$ and 
$\langle h[\dot{G_\mathcal{P}}](\dot{q}_n) : n \in \omega \rangle$. So $p*\dot{q}$ is a common lower bound 
for $\langle p_n*\dot{q}_n : n \in \omega \rangle$ and its image under $h.$

Now let $(*)$ be the assertion that the hypothesis below imply the conclusion below.

Hypothesis : 
\begin{enumerate}
 \item $\alpha<\beta \leq \delta$,
 \item $M,N,h,$ and $\langle p_n: n\in \omega \rangle $, are as in 
Definition \ref{S-cic} for $\mathcal{P}=\mathcal{P}_\delta$, with $\alpha \in M$,
\item $r \in \mathcal{P}_\alpha$ and $r_h \in \mathcal{P}_{h(\alpha)}$ are lower bounds for 
$\langle p_n \upharpoonright \alpha: n\in \omega \rangle $ and 
$\langle h(p_n \upharpoonright \alpha): n\in \omega \rangle $, respectively such that:
\begin{enumerate}
\item
$\mathrm{supp}(r) \subset M$, and $\mathrm{supp}(r_h) \subset N$, and
\item
$r(\xi)=r_h(\xi)$ for all $\xi$ in $M\cap N$.
\end{enumerate}
\end{enumerate}

Conclusion: There are lower bounds $\bar{r} \in \mathcal{P}_\beta$, $\bar{r}_h \in \mathcal{P}_{h(\beta)}$
for $\langle p_n \upharpoonright \beta: n\in \omega \rangle $ and $\langle h(p_n \upharpoonright \beta): n\in \omega \rangle $ 
respectively such that:
\begin{enumerate}
\item
$\mathrm{supp}(\bar{r}) \subset M$, and $\mathrm{supp}(\bar{r}_h) \subset N$,
\item
$\bar{r}(\xi)=\bar{r}_h(\xi)$ for all $\xi$ in $M\cap N$, and
\item
$\bar{r}\upharpoonright \alpha = r$ and $\bar{r}_h\upharpoonright h(\alpha)=r_h$.
\end{enumerate}

First note that $(*)$ is stronger than the lemma. To see this put $\alpha = 0$ and $\beta = \delta$, then any lower bound for 
$\bar{r}$ and $\bar{r_h}$ works as the desired condition.

We proceed by induction on $\beta$ to show $(*)$. The successor step is trivial by what we just proved.
If $\beta$ is a limit ordinal let
$\langle \alpha_i: i \in \omega \rangle$  be an increasing cofinal sequence in  
$M\cap \delta$ with $\alpha_0 = \alpha$. Then
$\langle h(\alpha_i) : i \in \omega \rangle$ is cofinal in $N\cap \delta$.
Let $\langle r^i : i \in \omega  \rangle$ and $\langle r_h^i : i \in \omega \rangle$ be sequences of conditions such that the following hold.
\begin{enumerate}
\item $r^i \in \mathcal{P}_{\alpha_i}$ is a lower bound for $\langle p_n \upharpoonright \alpha_i : n \in \omega \rangle$ and 
$r_h^i \in \mathcal{P}_{h(\alpha_i)}$ is a lower bound for  $\langle h(p_n) \upharpoonright h( \alpha_i) : n \in \omega \rangle$.
\item $\mathrm{supp}(r^i) = M\cap \alpha_i    $ and $\mathrm{supp}( r_h^i) = N \cap h(\alpha_i)  $.
\item If $\xi \in M\cap N$ then $r^i (\xi) = r_h^i (\xi)$.
\item If $i< j$ then $r^i = r^j \upharpoonright \alpha_i$ and $r_h^i = r_h^j \upharpoonright h(\alpha_i) .$
\end{enumerate}
Now $\bar{r} = \bigcup_{i \in \omega} r^i$ and $\bar{r}_h = \bigcup_{i \in \omega} r_h^i$ are as desired.
\end{proof}

We finish this section with a few remarks.
Unlike Lemma 2.4 of chapter VIII of \cite{proper_forcing}, in the last lemma there is no hypothesis on the 
length of the 
iteration. In other words by the lemmas in this section, as long as $\kappa$ is 
regular and greater than
the continuum, any countable support iteration of posets that have the $S$-cic for $\kappa$ has the 
$\kappa$-cc.

It is possible to define $S$-proper posets to be the ones which have $M$-generic condition $q$ below $p$,
whenever $M$ comes from a stationary set $S$, and $p$ is a condition inside $M$. These posets inherit 
many nice properties of proper posets. For instance, they preserve stationarity of all stationary subsets of 
$S$, and their countable support iterations do not add new cofinal branches to $\omega_1$-trees provided that the 
iterands have this property. 
$S$-properness is weaker than both $S$-completeness and properness.

\section{Minimal Kurepa Trees} \label{embedding forcing}

This section is devoted to the proof of Theorem \ref{main}. A fastness notion for closed unbounded 
subsets of $\omega_2$ is used in the definition of the forcings which add embeddings. 

\begin{defn}
A club $C_U \subset \omega_2$ is fast enough for $U$ and $T$ if it is the set of all 
$\sup(M_\xi \cap \omega_2)$ where $\langle M_\xi : \xi \in \omega_2 \rangle $
 is a continuous $\in$-chain
of $\aleph_1$-sized elementary submodels of $H_\theta$ 
such that:
\begin{itemize}
\item $U,T$ are in $M_0$,
\item  for all $\xi \in \omega_1$, $\xi \cup \omega_1 \subset M_\xi$, and 
\item for all $\xi \in \omega_1$,  $\langle M_\eta : \eta \leq \xi \rangle$ is in $M_{\xi + 1}$.
\end{itemize}
Let $\phi$ be  a partial function from $\omega_2$ to $\omega_2$ and $C \subset \omega_2$. 
We say $\phi$ respects $C$, whenever   $\xi < \alpha$ if and only if $\phi(\xi) < \alpha$, 
for all $\xi \in \dom(\phi)$ and  $\alpha \in C$.
\end{defn}

\begin{defn} \label{forcing definition}
Suppose $T$ is an everywhere Kurepa tree with 
$\mathcal{B}(T)=\langle b_\xi ; \xi \in \omega_2 \rangle$,
$U$ a downward closed everywhere Kurepa subtree of $T$, and
$C_U\subset \omega_2$ a club that is fast enough. 
The poset $\mathcal{Q}_{T,U}$ is the set of all conditions $p=(f_p,\phi_p)$ such that the following hold.
\begin{enumerate}
\item
$f_p:T\upharpoonright A_p \longrightarrow U \upharpoonright A_p$ is a level preserving tree isomorphism,
where $A_p \subset \omega_1$ is countable and closed with $\max A_p = \alpha_p$.
\item
$\phi_p$ is a countable partial injection from $\omega_2$ to $\omega_2$ such that:
\begin{enumerate}
\item for all $\xi \in \dom(\phi_p)$, $b_{\phi_p(\xi)} \in \mathcal{B}(U)$,
\item $\phi_p$ respects $C_U$.
 \end{enumerate}
\item
If $\xi \in \dom(\phi_p)$ then $f_p(b_\xi (\alpha_p))= b_{\phi_p(\xi)}(\alpha_p)$.
\end{enumerate}
We let $p\leq q$ if $f_q \subset f_p$, $f_p \upharpoonright T_{\leq \alpha_q} = f_q$ and $\phi_q \subset \phi_p$. 
Whenever there is no ambiguity we use $\mathcal{Q}$ instead of $\mathcal{Q}_{T,U}$.
\end{defn}

For more clarity we include the density argument we need.
For all $\alpha \in \omega_1$, the set of all 
conditions $q$ with $\alpha_q \geq \alpha $ is dense in $\mathcal{Q}$. 
In order to extend an arbitrary $r \in \mathcal{Q}$ to a condition $q$ with $\alpha_q > \alpha$,
let $\beta \in \omega_1$ be greater than $\alpha$ such that for all $t \in T_{\alpha_r}$ the set of all
$s > t $ with $ s \in T_\beta \setminus \{b_\xi (\beta): \xi \in \dom(\phi_r)\}$ is infinite and 
for all $u \in U_{\alpha_r}$ the set of all 
$v > u$ with $v \in U_\beta \setminus \{b_\eta (\beta) : \eta \in \range(\phi_r)\}$ is infinite.
Define $A_q = A_r \cup \{\beta\}$ and $\phi_q = \phi_r$.  
For each $t \in T_{\alpha_r}$, choose $g_t$ a bijection from the set
$\{s \in T_\beta: s > t \textrm{ and } s \notin \bigcup_{\xi \in \dom(\phi_r)} b_\xi \}$ to 
$\{v \in U_\beta: v > f_r(t) \textrm{ and } v \notin \bigcup_{\eta \in \range(\phi_r)} b_\eta \} $.
Now let $f_q = f_r \cup \bigcup_{t \in T_{\alpha_r}}g_t$ works.

Moreover, for all $\xi \in \omega_2$ the set $D_\xi$ consisting of all conditions $q$ with 
$\xi \in \dom(\phi_q) $ is dense in $\mathcal{P}$.
In order to see that, let
$p \in \mathcal{Q}$, and $ \xi \in \omega_2 \setminus \dom(\phi_p)$.
Also let $\langle M_\xi : \xi \in \omega_2 \rangle $ and $C_U$  be as above.
Note that by continuity, $ \min \{ \zeta : \xi \in M_\zeta \}$ has to be a successor ordinal so let
$\rho \in \omega_2$ such that
$\rho +1= \min \{ \zeta : \xi \in M_\zeta \}$. 
Recall that $U$ is everywhere Kurepa.
So there exists $\eta \in \omega_2$ such that
$f_p(b_\xi(\alpha_p)) = b_\eta(\alpha_p)$,
$\eta \notin \range(\phi_p)$,
$\rho +1= \min \{ \zeta : \eta \in M_\zeta \}$, and 
$b_\eta \in \mathcal{B}(U)$.
Now the condition defined by  $f_q = f_p$ and $\phi_q = \phi_p \cup \{(\xi, \eta)\}$ works. 
Similarly, for all $\xi$ with $b_\xi \in \mathcal{B}(U)$, the set of all conditions $q$ with $\xi \in \range(\phi_q)$ is dense. 

\begin{lem} \label{nice forcing}
Suppose $T$ is an everywhere Kurepa tree with $\Omega(T)=S$ stationary, $\mathcal{P}$ is an 
$S$-complete forcing, and $\dot{U}$ is a $\mathcal{P}$-name for a
downward closed everywhere Kurepa  subtree of $T$. Then 
\begin{itemize}
\item[1)]$\Vdash_{\mathcal{P}} ``\dot{\mathcal{Q}_{T,U}} $ is $\check{S}$-complete'', and 
\item[2)]$\Vdash_{\mathcal{P}} ``\dot{\mathcal{Q}_{T,U}} $ has the $\check{S}$-cic for $\check{\omega_2}$''.
\end{itemize}

\end{lem}
\begin{proof}
Let $G\subset \mathcal{P}$ be a $\mathbf{V}$-generic filter. Note that by Lemma \ref{No New branch},
since $\mathcal{P}$ is $S$-complete it does not add new cofinal
branches to $\omega_1$-trees. Moreover, $S \subset [\mathcal{B}(T)]^\omega$ is stationary in $\mathbf{V}[G]$ 
because $\mathcal{P}$ preserves the stationarity of stationary subsets of $S$.

Now we work in $\mathbf{V}[G]$. To see (1) assume $M$ is suitable for
 $\dot{\mathcal{Q}}_G=\mathcal{Q}$, 
and $M \cap \mathcal{B}(T) \in S$. Also let
$\langle p_n = (f_n,\phi_n): n\in \omega \rangle$ 
be a descending
 $(M,\mathcal{Q})$-generic sequence, and $\delta = M\cap \omega_1$.
 Now by a density argument and elementarity
 \begin{itemize}
 \item
 $\bigcup_{n \in \omega} \dom(\phi_n)=M\cap \omega_2$, and
 \item
 $\bigcup_{n\in \omega} \dom(f_n)=T\upharpoonright A,$ for some $A$ which is cofinal in $\delta$.  
 \end{itemize}
Let  $\phi_p = \bigcup_{n \in \omega}\phi_n $,
$A_p = \{ \delta \} \cup \bigcup_{n \in \omega} A_{p_n}$,
$f_p \upharpoonright (T_{< \delta})= \bigcup_{n \in \omega} f_n$, 
and for each $\xi \in M\cap \omega_2$ define $f_p(b_\xi(\delta))=b_{\phi_p(\xi)}(\delta)$. 
We show that $p=(f_p, \phi_p)$ is a condition in the poset $\mathcal{Q}_{T,U}$.
Note that by elementarity and the fact that $U$ is everywhere Kurepa, 
$\alpha_{\mathcal{B}(T)\cap M} = \alpha_{\mathcal{B}(U) \cap M} = \delta$.
First we show that $\dom(f_p)$ contains $T_\delta$.
Let $t \in T_\delta$. Since $M \cap \mathcal{B}(T) \in \Omega(T)$, there exists 
$b \in M \cap \mathcal{B}(T) $ such that $t \in b$. Let $\xi \in M$ be such that
$b_\xi =b$. Then by the density argument after Definition \ref{forcing definition},
there exists $n \in \omega$ such that $\xi \in \dom(\phi_n)$. So $t \in \dom(f_p)$.

In order to see $\range (f_p) \supset U_\delta$, let $u \in U_\delta$. 
Since $M \cap \mathcal{B}(T) \in \Omega(T)$ there exist $b \in \mathcal{B}(T) \cap M$ such that 
$u \in b$. Since $U$ is downward closed $M\vDash b \subset U$. So by elementarity 
$b \in \mathcal{B}(U) \cap M$. Let $\eta \in \omega_2$ be such that $b_\eta =b$.
Then by elementarity and the density statements after Definition \ref{forcing definition}, 
there exists $n \in \omega$ and $\xi \in \omega_2 \cap M$ such that $\phi_n(\xi)= \eta$. 
Therefore $f_p(b_\xi(\delta))= b_\eta(\delta)=b(\delta)=u$, as desired.
Other requirements of Definition \ref{forcing definition} trivially hold for $p$.
 
 For (2), still in $\mathbf{V}[G]$,  let $M,N,\langle p_n = (f_n, \phi_n): n\in \omega \rangle$,
 and $h$ be as in Definition \ref{S-cic}
with $M\cap \omega_1 = N\cap \omega_1 = \delta$ and  $\min((N\setminus M)\cap \kappa) > \sup(M\cap \kappa)$.
Also let  $\alpha_M= \min ((M\setminus N)\cap \kappa )$.
 We let $h(\phi_n)=\psi_n$ and since $h$ fixes the the elements of $M \cap N$, $h(f_n)=f_n$. Note that
 $b(\delta)=[h(b)](\delta)$, for all $b \in \mathcal{B}(T)\cap M$.
Let  $\phi_p = \bigcup_{n \in \omega}(\phi_n \cup \psi_n)$ and
 $f_p(b_\xi(\delta))= b_{\phi(\xi)}(\delta)$.

We need to show that $\phi_p$ is  one to one.
Obviously, if $\xi \neq \eta \neq h(\xi)$ then $\phi_p(\xi) \neq \phi_p(\eta)$.  
So  assume for a contradiction that $h(\xi) \neq \xi$ and  $\phi_p (\xi) = \phi_p (h(\xi))$. 
Fix $n \in \omega$ with $\xi \in \dom(\phi_n)$. Then $\phi_n(\xi)=\psi_n(h(\xi))=[h(\phi_n)](h(\xi))=h(\phi_n(\xi))$.
But $h$ fixes $\phi_n(\xi)$, so $\phi_n(\xi) \in M\cap N$. 
On the other hand  $\alpha_M \leq \xi$, since $h$ is an isomorphism which  fixes the elements of $M\cap N$.
Therefore 
\begin{center}
$\phi_n(\xi) < \sup (M\cap N \cap \omega_2)< \alpha_M \leq \xi $.
\end{center}
But $C_U \in M\cap N$ and $\sup(M\cap N \cap \omega_2) \in C_U$ which contradicts the fact that $\phi_n$ 
respects $C_U$.

Now as in the previous part, $p =(f_p ,\phi_p)$ is a common lower bound for 
$\langle p_n : n \in \omega \rangle$ and $\langle h(p_n) : n \in \omega \rangle$.
\end{proof}

\begin{rmk}
As mentioned in the proof of Lemma \ref{nice forcing}, in the second part we used the fact 
that the function $\phi_p$ of conditions $p$ in the forcing $\mathcal{Q}$ respect  a closed 
unbounded subset of $\omega_2$. Note that even one single instance of the forcing obtained by dropping this restriction 
collapses $\omega_2$. More precisely, the generic filter of such forcing adds a cofinal $\omega_1$-sequence to $\omega_2$ of the 
ground model.
\end{rmk}

The following proposition follows from \ref{chain CS}, \ref{iteration}, \ref{chain}, and \ref{nice forcing}.
\begin{prop}
Assume $\mathrm{GCH}$. If $T$ is an everywhere Kurepa tree with $\omega_2$ many cofinal branches such that
 $\Omega(T)\subset [\mathcal{B}(T)]^\omega$ is stationary,
then there is a forcing extension in which $\mathrm{GCH}$ is still  true and
$T$ is club isomorphic to all of its downward closed everywhere Kurepa subtrees.
\end{prop}

In order to prove Theorem \ref{main}, it suffices to find a model of 
$\mathrm{GCH}$ in which there is a Kurepa tree that 
satisfies the hypothesis of the proposition.

\begin{proof}[Proof of Theorem \ref{main}]
 Assume $\mathrm{CH}$.
Let $\mathcal{K}$ be the poset consisting of 
 conditions of the form $p=(T_p,b_p)$ where
\begin{itemize}
\item
$T_p$ is a countable tree of height $\alpha_p+1$ such that for all $t \in T_p$ there exists
 $s \in (T_p)_{\alpha_p}$ with $t<s$,
\item
$b_p$ is countable partial function from $\omega_2$ to the last level of $T_p$. 
\end{itemize}
 We let $p\leq q $ in $\mathcal{K}$ if 
\begin{itemize}
\item $(T_p)_{\leq \alpha_q} =T_q$,  
\item $\dom(b_p)\supset \dom(b_q)$, and
\item for all $\xi \in \dom(b_q)$,  $b_q(\xi) \leq b_p(\xi)$.
\end{itemize}

It is well known that $\mathcal{K}$ is countably closed and under $\mathrm{CH}$ it has the $\omega_2$-chain condition. 
If $T$ be the $\mathcal{K}$-generic tree  then
$\Omega(T)$ is stationary in $[\mathcal{B}(T)]^\omega$.
To see this, let $p$ be a condition and $\dot{E}$ be 
a $\mathcal{K}$-name 
for a club in $[\mathcal{B}(T)]^\omega$ which is forced by $p$ to be disjoint from $\Omega(T)$.
Let $M$ be suitable for $\mathcal{K}$ with $p, \dot{E}$, etc in $M.$
Then for any sequence $\langle p_n: n\in \omega \rangle $ which is $(M,\mathcal{K})$-generic
and $p_0\leq p$ we can form a lower bound $\bar{p}$ for the sequence such that
$\dom(\bar{p})=M\cap \omega_2$. Note that such a condition forces
that $M\cap \omega_2 =M[\dot{G}]\cap \omega_2$, where $\dot{G}$ is the canonical name for the 
$\mathcal{K}$-generic filter. On the other hand $\bar{p}$ forces that 
$M[\dot{G}]\cap \mathcal{B}(T)\in \dot{E}$, because it is $M$-generic.
So $\bar{p}$ forces that $M[\dot{G}]\cap  \mathcal{B}(T)\in \dot{E} \cap \Omega(T) $ 
which is a contradiction. 

In order to show that $T$ does not have any Aronszajn subtree in the final model, 
after embeddings are added, we will show 
$\Vdash_\mathcal{K} \dot{T}$ has no Aronszajn subtrees.
Note that this suffices by Lemma \ref{No A subtree}.
Let $\dot{U}$ be a $\mathcal{K}$-name for an uncountable downward closed subtree of $\dot{T}$,
where $\dot{T}$ is a $\mathcal{K}$-name for the tree $T$.
Let $M$ be a suitable model for $\mathcal{K}$ with $\dot{U}\in M$. By the assumptions,
for all $\xi \in \omega_2 \cap M$ and $p \in M\cap \mathcal{P}$ there is an extension
 $q \in M\cap \mathcal{P}$ of $p$ such that for some $\alpha \in \omega_1$,
 $q$ forces that $b_\xi(\alpha) \notin \dot{U}$.
  Here $b_\xi=\bigcup_{p\in \dot{G}}b_p(\xi)$
 and $\dot{G}$ is the canonical name for the generic filter of the forcing $\mathcal{K}$.
 By elementarity, $\alpha \in M\cap \omega_1$.
 Let $\langle p_n: n\in \omega \rangle $ be an $(M,\mathcal{K})$-generic sequence 
 such that for all $\xi \in \omega_2\cap M$ there is an $n \in \omega$ such that 
 $p_n \Vdash b_\xi (M\cap \omega_1) \notin \dot{U}$. Let $q$ be a lower bound for this sequence 
 such that $\dom(b_q)=M\cap \omega_2$. Then $q$ forces that $\dot{U}$ is countable,
 which is a contradiction.
Therefore, every downward closed subtree of $T$ contains $b_\xi$ 
 for some $\xi \in \omega_2$. Hence $T$ has no Aronszajn subtree and 
 $\langle b_\xi: \xi \in \omega_2 \rangle = \mathcal{B}(T)$.
 \end{proof}
 
 \section{Concluding remarks}

We finish this paper by a few remarks. Primarily, there are other Kurepa trees which can be made minimal 
by the method of this paper.
The Kurepa tree constructed in \cite{MR3611332} also satisfies the hypothesis of the last
proposition. So it can be made minimal in the same way as above. By the work in 
\cite{MR3611332}, this tree has no Aronszajn subtree. So by Lemma \ref{No A subtree}
this tree has no Aronszajn subtree after embeddings are added either.

The club minimality of an everywhere Kurepa tree clarifies the behavior of
 the invariant $\Omega$ introduced in \cite{no_real_Aronszajn}. 
%
We need the following definition which is due to Woodin.
\begin{defn}
Let $A,B$ be two collections of countable sets and $X=\bigcup A$,
$Y= \bigcup B$. We say $B \leq A$ 
if there is an injection $\iota : X \rightarrow Y $
such that for club many $M$ in $[Y]^{\omega}$, 
if $M \in B$ then $\iota ^{-1}M$ is in $ A$.
We let $B < A$ if $B\leq A$ but not $A\leq B$;
$A$ and $B$ are \emph{equivalent} ($A \equiv B$) if $A \leq B$ and $B \leq A$.
\end{defn}
\noindent
Assume $T$ is an $\omega_1$-tree which is equipped with a lexicographic order. Let $L$ be the linear order 
consisting of the elements of $T$ with the lexicographic order.
Then $\Omega(T)$ which we defined  is equivalent to $\Omega(L)$ defined in \cite{no_real_Aronszajn}. 
The invariant $\Omega$ played an essential role in the proof of the 
following results.
\begin{thm} \cite{no_real_Aronszajn}
Assume $\PFA^+$. If $L$ is a minimal non $\sigma$-scattered linear order, then it is either 
a real or Countryman type.
\end{thm}
\begin{thm}\cite{first}
If there is a supercompact cardinal then there is a forcing extension which satisfies $\mathrm{CH}$ in
which there is no minimal non $\sigma$-scattered linear order.
\end{thm}
The role of $\Omega$ in  \cite{no_real_Aronszajn} can be described as follows. 
First note that if $L_0 \subset L$ then $\Omega(L) \leq \Omega(L_0)$.
By the work in \cite{no_real_Aronszajn},  
$\Omega(L)$ contains a club 
iff 
$L$ is $\sigma$-scattered. 
Also for linear orders $L_0 \subset L$, $L$ does not embed in 
$L_0$ if $\Omega(L_0) > \Omega(L)$. 
Part of the work in
\cite{no_real_Aronszajn} and \cite{first} was to deduce from appropriate hypotheses that 
if $L$ is a non $\sigma$-scattered linear order that does not contain any real or Aronszajn 
type then there is $L_0\subset L$ such that $\Omega(L_0)$ does not contain a club and
$\Omega(L_0) > \Omega(L)$. Recall that a linear order is called a real type if it is isomorphic to an 
uncountable subset of the reals.
A linear order $L$ is said to be scattered if it has no copy of the rationals. 
It is called $\sigma$-scattered if it is a countable union of scattered suborders.

Note that if $L$ is a real type or Aronszajn type, for every $L_0 \subset L$ either $\Omega(L_0) \equiv \Omega(L)$ 
or else  $\Omega(L_0)$ contains a club.
One might ask if this characterizes all linear orders which are either real or Aronszajn types.
A consistent affirmative answer is given in \cite{no_real_Aronszajn} and \cite{first}. 
The existence of a club minimal 
Kurepa tree with no Aronszajn subtree  gives a consistent negative answer to this question.

\section*{Acknowledgments}
For continual support and encouragement, the author would like to thank Justin Tatch Moore.

The research presented in this paper was supported in part by NSF grants DMS-1262019 and DMS-1600635.

\def\Dbar{\leavevmode\lower.6ex\hbox to 0pt{\hskip-.23ex \accent"16\hss}D}

\end{document}